\documentclass[a4paper,12pt]{article}
\usepackage{hyperref}
\hypersetup{
    colorlinks=false,
    linktoc=all,
    linkcolor=red,     %choose some color if you want links to stand out
    citecolor=blue,     %choose some color if you want links to stand out
}

\usepackage{amsmath,amssymb,amsfonts,amsthm,graphics,enumitem}
\usepackage{latexsym, bm}
\usepackage{multicol}
\usepackage{indentfirst}

\usepackage{cite}

\usepackage{graphicx}
\usepackage{caption}
\usepackage{textcomp}
\usepackage{tikz}
\usetikzlibrary{calc,arrows,positioning,shapes}
\usepackage{bigdelim}
\usepackage{nicematrix}
\usepackage{mathtools}
\usepackage{comment}
\usepackage{makecell}
\usetikzlibrary{decorations.pathreplacing}
\usetikzlibrary{intersections}
\usepackage{enumerate}
\usepackage{graphicx,booktabs,multirow}
\usepackage{appendix}

\usepackage{setspace}
\newtheorem{theorem}{Theorem}[section]
\newtheorem{proposition}[theorem]{\rm\bfseries Proposition}

\newtheorem{lemma}[theorem]{Lemma}
\newtheorem{corollary}[theorem]{\rm\bfseries Corollary}
\newtheorem{remark}[theorem]{Remark}
\newtheorem{conjecture}[theorem]{Conjecture}
\newtheorem{problem}[theorem]{Problem}
\newtheorem{Proof of Theorem 1.7.}[theorem]{Proof of Theorem 1.7.}

\newcommand{\ignore}[1]{}
\begin{document}
\noindent
\begin{spacing}{1.1}

\title {Two problems on booksize and triangular edges in Nosal graphs}
\date{}

\author{
Xinghui Zhao\footnote{	School of Mathematical Sciences, South China Normal University,  Guangzhou, 510631, P. R. China. E-mail: {\tt xhzhao@m.scnu.edu.cn}. } \;\;   \;\; Lihua You\footnote{Corresponding author. School of Mathematical Sciences, South China Normal University, Guangzhou, 510631, P. R. China. 
E-mail: {\tt ylhua@scnu.edu.cn}.} \;\;   \;\; Jing Zeng\footnote{College of Cryptology and Cyber Science, Nankai University, Tianjin 300350, P. R. China. E-mail: \texttt{jingzeng@mail.nankai.edu.cn}.} \;\;   \;\;  Xiaoxue Zhang\footnote{School of Mathematical Sciences, South China Normal University, Guangzhou, 510631, P. R. China.
E-mail: {\tt zhang\_xx1209@163.com}.}
} 
\maketitle
\begin{abstract}
A graph $G$ with $m$ edges is said to be a Nosal graph if $\rho(G)>\sqrt{m}$. For a graph $G$, we write $bk(G)$ for its maximum book size and $\tau(G)$ for the number of edges contained in triangles. Li, Liu and Zhang [J. Combin. Theory Ser. B 179 (2026) 219--249] proved that every $m$-edge Nosal graph satisfies $bk(G)> \frac{1}{24}\sqrt{m}$ and $\tau(G) > \frac{1}{12}\sqrt{m}$. Recently, two results on the booksize constant are proved: $\frac{1}{9}$ by  Zhai, Li and Lou [arXiv:2601.10163v2], and $\frac{1}{4}$ by Chen, Li and Tang [arXiv:2607.16746v1].

In this paper, we establish the following result: Every $m$-edge graph $G$ with no isolated vertices and $\rho(G)\geq \sqrt{m}$ that is not isomorphic to any complete bipartite graph satisfies $bk(G)\geq\frac{\rho(G)}{3}$ and $\tau(G)\geq \rho(G)$. As direct consequences, we answer a question of Li, Liu and Zhang [J. Combin. Theory Ser. B 179 (2026) 219--249] and confirm a conjecture of Li, Feng and Peng [J. Graph Theory 110 (4) (2025) 408--425].

\end{abstract}


\section{Introduction}
A central theme in extremal graph theory is to determine how many copies of a prescribed substructure must occur under given global constraints. In spectral extremal graph theory, one is often interested in how lower bounds on the spectral radius, together with a prescribed number of edges, guarantee the occurrence of many copies of certain subgraphs. A classical result of Nosal \cite{Nosal}, also known as Nosal's theorem, states that every triangle-free graph $G$ with $m$ edges satisfies $\rho(G) \leq \sqrt{m}$, where $\rho(G)$ denotes the spectral radius of $G$. In 2007, Bollob\'as and Nikiforov \cite{bn2007} proved that $t(G)\geq \frac{n^2}{12}\left(\rho(G)-\frac{n}{2}\right)$ and $t(G)\geq \frac{\rho(G)}{3}\left(\rho^2(G)-m\right)$, where $t(G)$ is the number of triangles of $G$ and $n=|V(G)|$. The latter inequality was independently obtained by Cioab\u{a}, Feng, Tait and Zhang \cite{friendship2020}. 
Ning and Zhai \cite{NingZhaitri} later characterized the equality case, showing that equality holds if and only if $G$ is a complete bipartite graph possibly together with isolated vertices. 
They also proved that if $\rho(G)>\sqrt{m}$, then $t(G)\geq \left\lfloor \frac{\sqrt{m}-1}{2}\right\rfloor$, and this bound is best possible. These inequalities can be viewed as spectral supersaturation results: once the spectral radius exceeds the corresponding extremal threshold, the graph must contain many triangles. As noted in \cite{ly}, supersaturation refers to the phenomenon that exceeding an extremal threshold forces the appearance of many forbidden structures, rather than just a single one. Related results on supersaturation were obtained in \cite{LS75,LS83,Colorcri,PYSuperSa,HeMaYangC4,LPStri}.
% From this viewpoint, Nosal's theorem gives the spectral threshold for the appearance of a triangle in terms of the number of edges.

Following \cite{ly}, we say an $m$-edge graph $G$ is \textit{Nosal} if $\rho(G) > \sqrt{m}$. Several extensions of Nosal's theorem to other forbidden subgraphs $F$ have been studied, including cliques \cite{niki2002,bn2007,linningwuCPC,zhang04regular}, complete bipartite graphs \cite{BG2009,nikizarank}, cycles \cite{evencycles,liningJGT2023,lizhaishu2024,zhangcycles2024}, trees \cite{CDTsiam2023,tree2024}, friendship graphs \cite{friendship2020,Zfriendship2022}; see also \cite{py} for related spectral results on triangular edges. For a survey on spectral conditions for extremal graph problems, we refer the reader to \cite{LLL2022}. Nosal's theorem guarantees that every $m$-edge graph $G$ with $\rho(G)>\sqrt{m}$ contains a triangle. It is then natural to ask whether this spectral condition forces many triangles to share a common edge, that is, whether it forces a book of large size.
A book of size $k$ is a collection of $k$ triangles sharing a common edge, and we denote by $bk(G)$ the maximum size of a book contained in $G$. Zhai, Lin and Shu \cite{ZhaiLinShu2021} conjectured that every $m$-edge Nosal graph $G$ should have large booksize.
Nikiforov \cite{nf} first proved that every $m$-edge Nosal graph $G$ satisfies $bk(G)> \frac{1}{12}\sqrt[4]{m}$, and suggested that the exponent $\frac{1}{4}$ might be improved. Motivated by this, Li and Peng \cite{ly2} proposed the following conjecture.

\begin{conjecture}{\rm(\!\!\cite{nf,ly2})}\label{c0}
   For  every $m$-edge Nosal graph $G$, we have $bk(G)=\Omega(\sqrt{m})$.
\end{conjecture}

The conjecture was recently confirmed by Li, Liu and Zhang \cite{ly}. 
More precisely, they proved the following lower bound and showed that the order $\sqrt{m}$ is best possible.

\begin{theorem}{\rm(\!\!\cite{ly})}\label{m0}
     If $G$ is an $m$-edge Nosal graph, then $bk(G)> \frac{1}{24}\sqrt{m}$. Furthermore, there
exist $m$-edge Nosal graphs with no book of size larger than $(\frac{1}{3}+o(1))\sqrt{m}$.
\end{theorem}

Recall that an edge is called triangular if it is contained in a triangle. 
While the booksize measures how many triangles can be forced to share a common edge, the number of triangular edges measures how many edges are involved in triangles. Li, Feng and Peng \cite{py} studied spectral supersaturation problems involving triangular edges and established a spectral version of the Erd\H{o}s--Faudree--Rousseau theorem. They also posed the following conjecture.
\begin{conjecture}{\rm(\!\!\cite{py})}\label{c1}
    If $G$ is a graph with $m$ edges and $\rho(G)\geq \sqrt{m}$, then $G$ has at least $ \sqrt{m}$ triangular edges, unless $G$ is a complete bipartite graph with possibly some isolated vertices.
\end{conjecture}

The corresponding version of this conjecture for Nosal graphs was recorded in \cite{ly} as follows.
\begin{conjecture}{\rm(\!\!\cite{py,ly})}\label{c2}
    Every $m$-edge Nosal graph  has more than $ \sqrt{m}$ triangular edges.
\end{conjecture}

In fact, since a book of size $k$ contains $2k+1$ triangular edges, Theorem \ref{m0} implies the following lower bound.
\begin{theorem}{\rm(\!\!\cite{ly})}\label{m1}
     Every $m$-edge Nosal graph  has at least $ \frac{1}{12}\sqrt{m}$ triangular edges.
\end{theorem}

Theorem \ref{m0} shows that no lower bound of the form $bk(G)\geq c\sqrt{m}$ can hold for all $m$-edge Nosal graphs with $c>\frac{1}{3}$. It is therefore natural to ask whether $\frac{1}{3}$ is the best possible constant. Li, Liu and Zhang \cite{ly} posed the following problem.
\begin{problem}{\rm(\!\!\cite{ly})}\label{P1}
    Does every $m$-edge Nosal graph contain a book of size $\frac{1}{3}\sqrt{m}$?
\end{problem}

Let $B_{r+1}$ denote the book of size $r+1$. Zhai, Li and Lou \cite{zm} made further progress on this problem by proving a spectral extremal result for $B_{r+1}$-free graphs. As an immediate consequence, their result implies that every $m$-edge Nosal graph satisfies $bk(G)>\frac{1}{9}\sqrt{m}$.
\begin{theorem}{\rm(\!\!\cite{zm})}\label{m3}
    Let $r$ be a positive integer, and let $G$ be a $B_{r+1}$-free graph with $m$ edges, where $m\geq (9r)^2$. Then $\rho (G)\leq \sqrt{m}$, with equality if and only if $G$ is a complete bipartite graph with possibly some isolated vertices.
\end{theorem}

\begin{corollary}{\rm(\!\!\cite{zm})}\label{co1}
     Every $m$-edge Nosal graph $G$ satisfies $bk(G)>\frac{1}{9}\sqrt{m}$.
\end{corollary}

Recently, Chen, Li and Tang \cite{CLT2023} further improved this result.

\begin{theorem}{\rm(\!\!\cite{CLT2023})}\label{mm1}
    Every $m$-edge Nosal graph $G$ satisfies $bk(G)>\frac{1}{4}\sqrt{m}$.
\end{theorem}

% \begin{theorem}\label{t1}
%     Let $G$ be an $m$-edge Nosal graph. Then $bk(G)>\frac{\rho(G)}{4}\geq\frac{\sqrt{m}}{4}$.
% \end{theorem}

%\begin{theorem}\label{t2}
 %   Let $G$ be an $m$-edge connected graph with $\rho(G)\geq \sqrt{m}$.  If $G$ is not isomorphic to a complete bipartite graph, then
 %   $\tau(G)\geq \rho (G)$.
%\end{theorem}

% \begin{theorem}\label{t2}
%     Let $G$ be an $m$-edge graph with no isolated vertices such that $\rho(G)\geq \sqrt{m}$ and $G$ is not isomorphic to any complete bipartite graph. Then $\tau(G)\geq \rho(G)$.
% \end{theorem}
 
 %For $X\subseteq V(G)$, $G[X]$ denote the subgraph induced by $X$. For two subsets $X,Y\subseteq V(G)$, we use $E(X,Y)$ to denote the set of edges with one vertex in $X$ and the other in $Y$, and we write $e(X,Y)=|E(X,Y)|$. To be written...

However, the optimal constant problem posed in Problem \ref{P1} remained open.  Moreover, Conjecture \ref{c1} on triangular edges was still unresolved.  In this paper, we prove Conjecture \ref{c1} and answer Problem  \ref{P1} in the affirmative.

\subsection{Main results}
Let $T(G)$ denote the set of all triangular edges of $G$, and let $\tau(G)=|T(G)|$. Our main results are the following two theorems.

% give an affirmative answer to Problem \ref{P1} and confirm Conjecture \ref{c2}. More precisely, we prove the following stronger statements under the condition $\rho(G)\geq \sqrt{m}$, with complete bipartite graphs as the only exceptions.

\begin{theorem}\label{t1}
   Let $G$ be an $m$-edge graph with no isolated vertices such that $\rho(G)\geq \sqrt{m}$ and $G$ is not isomorphic to any complete bipartite graph. Then $bk(G)\geq\frac{\rho(G)}{3}$.
\end{theorem}

 \begin{theorem}\label{t2}
    Let $G$ be an $m$-edge graph with no isolated vertices such that $\rho(G)\geq \sqrt{m}$ and $G$ is not isomorphic to any complete bipartite graph. Then $\tau(G)\geq \rho(G)$.
\end{theorem}

Theorem \ref{t1} answers Problem \ref{P1} in the affirmative. 
In fact, it proves a stronger statement under the condition $\rho(G)\geq \sqrt{m}$, with the only exceptions being complete bipartite graphs, possibly with some isolated vertices. Theorem \ref{t2} confirms Conjecture \ref{c1}, and hence also confirms Conjecture \ref{c2}. 

The lower bounds in Theorems \ref{t1} and \ref{t2} are stated in terms of 
$\rho(G)$, not just $\sqrt{m}$. This stronger spectral form is essential to the proof, as retaining $\rho(G)$ throughout the arguments is what allows us to reach the final conclusion.

%The lower bounds in Theorems \ref{t1} and \ref{t2} are given in terms of $\rho(G)$, not merely $\sqrt{m}$. This stronger spectral formulation is crucial to  the proofs, for it is the retention of $\rho(G)$ throughout the arguments that enables us to reach the final conclusions. %Hence, working directly with $\rho(G)$ is key to obtaining the optimal lower bounds for both problems.
 
 \subsection{Notation}

 As usual, we use $V(G)$ and $E(G)$ to denote the vertex set and the edge set of a graph $G$, respectively, and we denote $|E(G)|$ by $e(G)$. Let $v \in V(G)$, and let $U, W \subseteq V(G)$. We denote by $N_G(v)$ ($N(v)$ for short) the set of vertices adjacent to $v$ in $G$, and set $d_G(v) = |N_G(v)|$. When there is no risk of confusion, we write $N_U(v)$ and $d_U(v)$ for $N_G(v) \cap U$ and $|N_G(v) \cap U|$, respectively. Furthermore, for a set $E_1 \subseteq E(G)$, let $d_{E_1}(v)=\{u\in N_G(v)\mid uv\in E_1\}$. Finally, let $E_G(U,W)$ denote the set of edges with one endpoint in $U$ and the other in $W$, and write $e_G(U,W) = |E_G(U,W)|$.% In particular, if $U=W$, we simply write $e_G(U,W)$ as $e_G(U)$.
 
 %we use $\overline{G}$ to denote its complement. 
%For a set $E\subseteq E(\overline{G})$, let $G+E$ be the graph obtained by adding all edges in $E$, and for a set $V\subseteq V(G)$, let $G-V$ be the graph obtained by deleting the vertices in $V$ (and their incident edges). When $E=\{e\}$, we simply write $G+e$. For convenience, we write $\{1,\ldots,n\}$ as $[n]$. 

\subsection{Organization}\label{subsec-org}
In Section \ref{sec-pre}, we introduce some lemmas that will be used in our proofs. In Section \ref{sec-pft1}, we prove Theorem \ref{t1}. In Section \ref{sec-pft2}, we prove Theorem \ref{t2}.  Finally, in Section \ref{sec-con}, we  pose several further problems.

	\section{Preliminaries}\label{sec-pre}
	\begin{lemma}{\rm(\!\!\cite{dc})}\label{l1}
		Let $M$ be an irreducible non-negative symmetric matrix. Then the largest eigenvalue of $M$ corresponds to an eigenvector whose entries are all positive.	
	\end{lemma}

   % \begin{lemma}{\rm(\!\!\cite{dc})}\label{ll1}
	%	Let $G$ be a graph. Then $\rho(G)\leq \Delta(G)$.
	%\end{lemma}

    \begin{lemma}\label{l4}
        Let $a,b$ be real numbers with $0\leq a,b\leq 1$ and $a+b\geq 1$, and let $c=a\max\{2b-1,0\}+b\max \{2a-1,0\}$.  Then $a+b-1\leq c$.
    \end{lemma}
    \begin{proof}
        If $a,b\geq \frac{1}{2}$, then $a+b-1-c=-(2a-1)(2b-1)\leq0$.
If $a\geq \frac{1}{2}>b$, then $a+b-1-c=(a-1)(1-2b)\leq0$.
Similarly, if $b\geq \frac{1}{2}>a$, then $a+b-1-c\leq0$.

This completes the proof.
    \end{proof}

    \begin{lemma}\label{l7}
        Let $a,b,c,d$ be real numbers with $a,b\geq1$, $c\geq0$ and $\frac{1}{2}\leq d\leq1$, and let $R(a,b,c)=2\max \{a,b\}+\min\{a,b\}+2c$. Then $a+c+2b(2d-1)+1-d\leq dR(a,b,c)$.
    \end{lemma}
\begin{proof}
Clearly, we have $$dR(a,b,c)-(a+c+2b(2d-1)+1-d)=\begin{cases}
	(2d-1)(a-b+c)+(1-d)(b-1)\geq0,&\text{if}\ a\geq b; \\
	
	(1-d)(2b-a-1)+c(2d-1)\geq0,&\text{if}\ b>a.
\end{cases}$$

    %If $a\geq b$, then $R(a,b,c)-(a+c+2b(2d-1)+1-d)=d(2a+b+2c)-(a+c+2b(2d-1)+1-d)=(2d-1)(a-b+c)+(1-d)(b-1)\geq0$.

    %If $b>a$, then $R(a,b,c)-(a+c+2b(2d-1)+1-d)=(1-d)(2b-a-1)+c(2d-1)\geq0$.

            This completes the proof.
\end{proof}

 \section{The proof of Theorem \ref{t1}}\label{sec-pft1}
 Let $G$ be an $m$-edge graph with no isolated vertices such that $\rho(G)\geq \sqrt{m}$ and $G$ is not isomorphic to any complete bipartite graph. If $G$ is disconnected, then there exists a connected component $H$ of $G$ such that $\rho(H)=\rho(G)\geq\sqrt{m}$ and $bk(H)\leq bk(G)$. Clearly, $H$ is not a complete bipartite graph. Otherwise, $\sqrt{m}\leq \rho(H)=\sqrt{e(H)}$, and thus $e(H)=m$, which implies that $G$ contains an isolated vertex, a contradiction. Henceforth, we always assume that $G$ is connected.  For convenience,  we write $\rho=\rho(G)$. Then there exists a positive eigenvector $\textbf{x}=(x_v)_{v\in V(G)}$  corresponding to $\rho$  by Lemma \ref{l1}.

For any  $e=uv\in E(G)$, let $C_e=N(u)\cap N(v)$, $t_e=\sum\limits_{w\in C_e}x_w$, $r_e=\sum\limits_{w\in (N(u)\cup N(v))\setminus (C_e\cup\{u,v\})}x_w$ and $z_e=\sum\limits_{w\in V(G)\setminus (N(u)\cup N(v))}x_w$.  Let $\mathcal{K}_3(G)$ denote the set of all triangles in $G$. Then \begin{align}
    \sum\limits_{e\in E(G)}t_e= \sum\limits_{\mathcal{T}\in \mathcal{K}_3(G)}\sum\limits_{w\in V(\mathcal{T})}x_w.\label{b1}
\end{align}

\begin{lemma}\label{l91}
    $\sum\limits_{e\in E(G)}z_e= (m-\rho^2)\sum\limits_{v\in V(G)}x_v+ \sum\limits_{e\in E(G)}t_e$.
\end{lemma}
\begin{proof}
    For any  $e=uu'\in E(G)$, we have $\rho (x_u+x_{u'})=x_u+x_{u'}+2t_e+r_e$ and $\sum\limits_{v\in V(G)}x_v=x_u+x_{u'}+t_e+r_e+z_e$, which implies that $z_e=\sum\limits_{v\in V(G)}x_v-\rho (x_u+x_{u'})+t_e$. Thus we have \begin{align*}
\sum\limits_{e\in E(G)}z_e&=m\sum\limits_{v\in V(G)}x_v-\rho \sum\limits_{uu'\in E(G)}(x_u+x_{u'})+\sum\limits_{e\in E(G)}t_e\\
&=m\sum\limits_{v\in V(G)}x_v-\rho^2\sum\limits_{v\in V(G)}x_v+\sum\limits_{e\in E(G)}t_e\\
&=(m-\rho^2)\sum\limits_{v\in V(G)}x_v+ \sum\limits_{e\in E(G)}t_e.
    \end{align*}

    This completes the proof.
\end{proof}

\begin{lemma}\label{l15}
    $\mathcal{K}_3(G)\neq \emptyset$.
\end{lemma}
\begin{proof}
    	Suppose to the contrary that $\mathcal{K}_3(G)= \emptyset$.  Then $ \sum\limits_{e\in E(G)}t_e=0$ by \eqref{b1}, and thus $0\leq\sum\limits_{e\in E(G)}z_e= (m-\rho^2)\sum\limits_{v\in V(G)}x_v\leq0$ by Lemma \ref{l91} and $\rho\geq \sqrt{m}$. This implies $\sum\limits_{e\in E(G)}z_e=0$, that is, $z_e=0$ for any $e\in E(G)$, and $V(G)=N(u)\cup N(v)$ for any $e=uv\in E(G)$. Then $G$ is a complete bipartite graph by the definition of $z_e$, a contradiction.  Therefore, $\mathcal{K}_3(G)\neq \emptyset$. 
\end{proof}

For any $\mathcal{T}\in \mathcal{K}_3(G)$, let $Y_i(\mathcal{T})=\sum\limits_{v\notin V(\mathcal{T}),|N(v)\cap V(\mathcal{T})|=i }x_v$. Then $i\in \{0,1,2,3\}$ and \begin{align}
    \sum\limits_{\mathcal{T}\in \mathcal{K}_3(G)}Y_1(\mathcal{T})&=\sum\limits_{uv\in E(G)}\sum\limits_{w\in V(G)\setminus (N(u)\cup N(v))}|N(w)\cap C_{uv}|x_w\nonumber\\
    &\leq \sum\limits_{uv\in E(G)}|C_{uv}|z_{uv}\nonumber\\&\leq bk(G)\sum\limits_{uv\in E(G)}z_{uv}.\label{b4}
\end{align}
We now give a lower bound for $\sum\limits_{\mathcal{T}\in \mathcal{K}_3(G)}Y_1(\mathcal{T})$.

\begin{lemma}\label{l92}
    $\sum\limits_{\mathcal{T}\in \mathcal{K}_3(G)}Y_1(\mathcal{T})\geq (\rho-2bk(G))\sum\limits_{e\in E(G)}t_e$.
\end{lemma}
\begin{proof}
    Clearly, $\rho\sum\limits_{v\in V(\mathcal{T})}x_v=2\sum\limits_{v\in V(\mathcal{T})}x_v+Y_1(\mathcal{T})+2Y_2(\mathcal{T})+3Y_3(\mathcal{T})$ for each $\mathcal{T}\in \mathcal{K}_3(G)$. Thus we have \begin{equation}
        Y_1(\mathcal{T})\geq (\rho-2)\sum\limits_{v\in V(\mathcal{T})}x_v-2(Y_2(\mathcal{T})+3Y_3(\mathcal{T})),\label{b2}
    \end{equation}
where $Y_2(\mathcal{T})+3Y_3(\mathcal{T})=\sum\limits_{e\in E(\mathcal{T})}\sum\limits_{w\in C_e\setminus V(\mathcal{T})}x_w$. Combining $|C_e|\leq bk(G)$ and $t_e\geq0$, we obtain \begin{align*}
    \sum\limits_{\mathcal{T}\in \mathcal{K}_3(G)}(Y_2(\mathcal{T})+3Y_3(\mathcal{T}))&=\sum\limits_{\mathcal{T}\in \mathcal{K}_3(G)}\sum\limits_{e\in E(\mathcal{T})}\sum\limits_{w\in C_e\setminus V(\mathcal{T})}x_w\\&=\sum\limits_{e\in E(G)}\sum\limits_{w\in C_{e}}(t_e-x_w)\\
    &=\sum\limits_{e\in E(G)}(|C_e|-1)t_e\\
    &\leq (bk(G)-1)\sum\limits_{e\in E(G)}t_e.
\end{align*}

Combining \eqref{b1} and \eqref{b2}, we have \begin{align*}
    \sum\limits_{\mathcal{T}\in \mathcal{K}_3(G)}Y_1(\mathcal{T})&\geq (\rho-2) \sum\limits_{\mathcal{T}\in \mathcal{K}_3(G)}\sum\limits_{v\in V(\mathcal{T})}x_v-2 \sum\limits_{\mathcal{T}\in \mathcal{K}_3(G)}(Y_2(\mathcal{T})+3Y_3(\mathcal{T}))\\
    &\geq (\rho-2) \sum\limits_{e\in E(G)}t_e-2(bk(G)-1)\sum\limits_{e\in E(G)}t_e\\&=(\rho-2bk(G))\sum\limits_{e\in E(G)}t_e
\end{align*}
as desired.
\end{proof}

\noindent\textbf{\textit{Proof of Theorem \ref{t1}.}} By \eqref{b4}, Lemmas \ref{l91}, \ref{l15} and \ref{l92}, we have $(\rho-2bk(G))\sum\limits_{e\in E(G)}t_e\leq bk(G)\sum\limits_{e\in E(G)}z_e\leq bk(G)((m-\rho^2)\sum\limits_{v\in V(G)}x_v+ \sum\limits_{e\in E(G)}t_e)$, and thus $(\rho-3bk(G))\sum\limits_{e\in E(G)}t_e\leq bk(G)(m-\rho^2)\sum\limits_{v\in V(G)}x_v\leq 0$. Therefore, $bk(G)\geq\frac{\rho}{3}=\frac{\rho(G)}{3}$ since $\sum\limits_{e\in E(G)}t_e>0$.
This completes the proof. $\hfill\square$

\vspace{0.5cm}

From the above proof, we immediately obtain the following corollary.

\begin{corollary}
     Every  Nosal graph $G$ satisfies $bk(G)>\frac{\rho(G)}{3}$.
\end{corollary}

\begin{remark}
    To prove Theorem \ref{t1}, we use a weighted double-counting argument, giving both upper and lower bounds on the total weight of such vertices that are adjacent to exactly one vertex of a triangle. The upper bound arises from a combinatorial double-counting argument, while the lower bound is based on summing the Perron eigenvalue equations over each triangle and then combining this with  $bk(G)\geq |C_e|$. A comparison of the two bounds yields $bk(G)\geq \frac{\rho(G)}{3}$.
\end{remark}

 \section{The proof of Theorem \ref{t2}}\label{sec-pft2}

% Let $G$ be an $m$-edge Nosal graph. If $G$ is disconnected, then there exists a connected component $H$ of $G$ such that $\rho(H)=\rho(G)>\sqrt{m}$. Thus $H$ is a Nosal graph with $bk(H)<bk(G)$ and $\tau(H)<\tau(G)$. Henceforth, we always assume that $G$ is a connected Nosal graph. 
 
 Let $G$ be a connected $m$-edge  graph. Then
 there exists a positive eigenvector $\textbf{x}=(x_v)_{v\in V(G)}$ with $\max\limits_{v\in V(G)}x_v=1$ corresponding to $\rho(G)$  by Lemma \ref{l1}. Let $u^*\in V(G)$ with $x_{u^*}=\max\limits_{v\in V(G)}x_v=1$,  $U=N_G(u^*)$ and $W=V(G)\setminus (\{u^*\}\cup U)$. Then $\rho(G)=\rho(G)x_{u^*}=\sum\limits_{u\in U}x_u$. Thus we have \begin{align}
     \rho^2(G)=\sum\limits_{u\in U}\rho(G)x_u=|U|+\sum\limits_{uv\in E(G[U])}(x_u+x_v)+\sum\limits_{w\in W}d_U(w)x_w.\label{s1}
 \end{align}

 Note that $m=|U|+e(G[U])+e_G(U,W)+e(G[W])$, $e_G(U,W)=\sum\limits_{w\in W}d_U(w)$ and $e(G[W])=\frac{1}{2}\sum\limits_{w\in W}d_W(w)$. Combining \eqref{s1}, we have 
 \begin{align}
      \rho^2(G)&=m-e(G[U])-e_G(U,W)-e(G[W])+\sum\limits_{uv\in E(G[U])}(x_u+x_v)+\sum\limits_{w\in W}d_U(w)x_w\nonumber\\
      &=m+\sum\limits_{uv\in E(G[U])}(x_u+x_v-1)-\sum\limits_{w\in W}(d_U(w)(1-x_w)+\frac{1}{2}d_W(w))\nonumber\\
      &=m+\sum\limits_{uv\in E(G[U])}(x_u+x_v-1)-\sum\limits_{w\in W}f(w),\label{s2}
 \end{align}
where $f(w)=d_U(w)(1-x_w)+\frac{1}{2}d_W(w)$ for each $w\in W$.

Let $S_w=\sum\limits_{v\in U\setminus N_G(w)}x_v$ and $Z_w=\sum\limits_{z\in N_W(w) }x_z$ for each $w\in W$. Then \begin{align}
    \rho(G)x_w=\sum\limits_{v\in N_U(w) }x_v+Z_w=\rho(G)-S_w+Z_w.\label{s4}
\end{align} Hence, we have \begin{align}
    S_w-Z_w=\rho(G)(1-x_w)\geq 0.\label{s3}
\end{align}

\begin{lemma}\label{l2}
    $\sum\limits_{w\in W }f(w)\geq \frac{1}{2}\sum\limits_{v\in U }(x_v\sum\limits_{w\in W\setminus N_G(v) }x_w)$.
\end{lemma}
\begin{proof}
    Since $0\leq x_v\leq1$ for each $v\in V(G)$, $d_U(w)\geq \sum\limits_{v\in N_U(w) }x_v=\rho(G)-S_w$ and $d_W(w)\geq Z_w$  for each $w\in W$. Combining \eqref{s4}, \eqref{s3} and $\rho(G)\geq S_w$, we have 
    \allowdisplaybreaks
    \begin{align*}
        f(w)&=d_U(w)(1-x_w)+\frac{1}{2}d_W(w)\\
        &\geq (\rho(G)-S_w)\frac{S_w-Z_w}{\rho(G)}+\frac{1}{2}Z_w\\
        &=(\rho(G)-S_w)\frac{S_w-Z_w}{\rho(G)}+\frac{1}{2}Z_w+\frac{x_wS_w}{2}-\frac{S_w}{2}(1-\frac{S_w-Z_w}{\rho(G)})\\
        &=\frac{x_wS_w}{2}+\frac{(S_w-Z_w)(\rho(G)-S_w)}{2\rho(G)}\\
        &\geq \frac{x_wS_w}{2}.
    \end{align*}
    Thus 
        $\sum\limits_{w\in W }f(w)\geq\sum\limits_{w\in W }\frac{x_wS_w}{2}
        =\frac{1}{2} \sum\limits_{w\in W } (x_w\sum\limits_{v\in U\setminus N_G(w)}x_v)
        =\frac{1}{2} \sum\limits_{v\in U } (x_v\sum\limits_{w\in W\setminus N_G(v)}x_w)$.

    This completes the proof.
\end{proof}

Let $\mathcal{B}=\{uv\in E(G[U])\mid x_u+x_v\geq1\}$, and let $d_{\mathcal{B}}(v)=|\{u\mid uv\in \mathcal{B}\}|$. 
Let   \begin{equation*}
			 	\beta_v= 
			 	\begin{cases}
			 		\max\{x_u\mid uv\in\mathcal{B}\},&\text{if}\ d_{\mathcal{B}}(v)>0;\\
			 		0,&\text{if }\ d_{\mathcal{B}}(v)=0,
			 	\end{cases}
			 \end{equation*} and $\gamma_v=\max\{2\beta_v-1,0\}$.
Then we have the following lemma.
\begin{lemma}\label{l5}
    If $\mathcal{B}\neq\emptyset$, then $\sum\limits_{uv\in \mathcal{B}}(x_u+x_v-1)\leq \sum\limits_{v\in U} d_{\mathcal{B}}(v)x_v\gamma_v$.
\end{lemma}
\begin{proof}
    Clearly,  $0<x_u,x_v\leq1$ and $x_u+x_v\geq1$  for each edge $uv\in \mathcal{B}$. By Lemma \ref{l4}, we have $x_u+x_v-1\leq x_u\max\{2x_v-1,0\}+x_v\max\{2x_u-1,0\}$. On the other hand, we have $\gamma_v\geq \max\{2x_u-1,0\}$ and $\gamma_u\geq \max\{2x_v-1,0\}$. Thus  $\sum\limits_{uv\in \mathcal{B}}(x_u+x_v-1)\leq\sum\limits_{uv\in \mathcal{B}}(x_u\gamma_u+x_v\gamma_v)=\sum\limits_{v\in U} d_{\mathcal{B}}(v)x_v\gamma_v$.

    This completes the proof.
\end{proof}

\begin{comment}
    \begin{proposition}\label{p1}
    If $\mathcal{B}\neq\emptyset$ and $\sum\limits_{w\in W\setminus N_G(v)}x_w\geq 2d_{\mathcal{B}}(v)\gamma_v$ for each $v\in U$, then $\rho(G)\leq \sqrt{m}$.
\end{proposition}
\begin{proof}
    By Lemma \ref{l2} and Lemma \ref{l5}, we have 
        $\sum\limits_{w\in W}f(w)\geq  \frac{1}{2}\sum\limits_{v\in U }x_v\sum\limits_{w\in W\setminus N_G(v) }x_w
        \geq \sum\limits_{v\in U }d_{\mathcal{B}}(v)x_v\gamma_v
        \geq \sum\limits_{uv\in \mathcal{B}}(x_u+x_v-1)$.
    Combining $x_u+x_v-1<0$ for each $uv\in E(G[U])\setminus \mathcal{B}$ and \eqref{s2}, we have $ \rho^2(G)=m+\sum\limits_{uv\in \mathcal{B}}(x_u+x_v-1)-\sum\limits_{w\in W}f(w)+\sum\limits_{uv\in E(G[U])\setminus\mathcal{B}}(x_u+x_v-1)\leq m$.

    Therefore, $\rho(G)\leq \sqrt{m}$.
\end{proof}
\end{comment}

 For $uv \in \mathcal{B}$, we define $c_{uv}=|N_W(u)\cap N_W(v)|$ and $R(d_U(u),d_U(v),c_{uv})=2\max\{d_U(u),$ $d_U(v)\}+\min\{d_U(u),d_U(v)\}+2c_{uv}$.
\begin{lemma}\label{l6}
    Let $uv\in\mathcal{B}$. Then $\tau(G)\geq R(d_U(u),d_U(v),c_{uv})$.
\end{lemma}
\begin{proof}
    For any $z\in N_U(u)\cup N_U(v)$, $u^*z$ is a triangular edge, where $u^*z\in E_G(\{u^*\},U)$. For any $z'\in N_U(u)$ and  $z''\in N_U(v)$, $z'u$ and $z''v$ are triangular edges, where $z'u,z''v\in E(G[U])$.
    For any $w\in N_G(u)\cap N_G(v)\cap W$, $uw$ and $vw$ are  triangular edges, where $uw,vw\in E_G(U,W)$. 

Therefore, we have \begin{align}
    \tau(G)&\geq | N_U(u)\cup N_U(v)|+(|N_U(u)|+|N_U(v)|-1)+2c_{uv}\nonumber\\ &\geq (\max\{d_U(u),d_U(v)\}+1)+(d_U(u)+d_U(v)-1)+2c_{uv}\label{b6}\\&=R(d_U(u),d_U(v),c_{uv})\nonumber
\end{align}
as desired. 
\end{proof}

\noindent\textbf{\textit{Proof of Theorem \ref{t2}.}}   For convenience,  we write $\rho=\rho(G)$ in the proof. If $G$ is disconnected, then there exists a connected component $G_0$ such that $\rho(G_0)=\rho$, and thus $\rho(G_0)=\rho\geq \sqrt{m}$. Clearly, 
 $G_0$ is not a complete bipartite graph. Otherwise,  $\rho(G_0)=\sqrt{e(G_0)}$. Then we have $\sqrt{m}\leq \rho(G_0)=\sqrt{e(G_0)}\leq \sqrt{m}$, which implies that $e(G_0)=m$. Thus $G$ contains isolated vertices,  a contradiction. Hence we have $\tau(G)\geq \tau (G_0)$ and $G_0$ is not a complete bipartite graph. Accordingly, it is sufficient for us to establish the proof when $G$ is a connected graph.

If $E(G[U])=\emptyset$ or $\mathcal{B}=\emptyset$, then by \eqref{s2},  $G$ is a complete bipartite graph or $\rho< \sqrt{m}$,  which contradicts our assumption.

\begin{comment}
    \noindent\textbf{Claim 1.} If $E(G[U])=\mathcal{B}$, $x_u=x_v=\frac{1}{2}$ for any $uv\in E(G[U])$, and $W=\emptyset$ or $f(w)=d_U(w)(1-x_w)+\frac{1}{2}d_W(w)=0$ for any $w\in W$, then $\tau(G)\geq \rho$.
\begin{proof}
Let $U'=\{v\in U\mid d_U(v)\geq1\}$. Then $|U'|\geq\
2$ since $E(G[U])\neq\emptyset$.

    If $W=\emptyset$, then $V(G)=\{u^*\}\cup U$. For any $v\in  U'$, from the characteristic equation at $v$,   it follows that $\frac{1}{2}\rho =1+\frac{d_U(v)}{2}$. Then $G[U']$ is a $(\rho-2)$-regular graph, and thus $\tau(G)\geq e_G[\{u^*\}, U']+\frac{(\rho-2)|U'|}{2}=\frac{\rho|U'|}{2}\geq \rho$.

    If $W\neq\emptyset$ and  $f(w)=0$ for any $w\in W$, then $d_W(w)=0$ and $d_U(w)(1-x_w)=0$. Since $d_W(w)=0$ for any $w\in W$ and $G$ is connected, we have $d_U(w)>0$, and thus $x_w=1$ for any $w\in W$.  From the characteristic equations at $u^*$ and $w$, it follows that $\rho=\sum\limits_{z\in U}x_z=\sum\limits_{z\in N_U(w)}x_z$. Then $N_G(w)= U$ for any $w\in W$  since $x_z>0$ for any $z\in U$, and thus $G\cong (1+|W|)K_1\vee G[U]$. For any $v\in U'$, we have $\frac{\rho}{2}=1+|W|+\frac{d_{U'}(v)}{2}$. Then $G[U']$ is a $(\rho-2-2|W|)$-regular graph, and thus $\tau(G)=e_G[\{u^*\}\cup W, U']+e(G[U'])=(1+|W|)|U'|+\frac{(\rho-2-2|W|)|U'|}{2}=\frac{\rho |U'|}{2}\geq\rho$.

     This completes the proof.
\end{proof}
\end{comment}

In what follows, assume that $E(G[U])\neq\emptyset$ and $\mathcal{B}\neq\emptyset$. 
Let $U_1$ be a set of all vertices incident to an edge in $\mathcal{B}$, and  $U_2=\{v\in U_1\mid \gamma_v=0\}$. Then $U_2\subseteq U_1$. Now we complete the proof of Theorem \ref{t2} by the following two cases.

\textbf{Case 1.} $U_2=U_1$.

For any $v\in U_1$,  there exists $u\in N_{U_1}(v)$ such that $x_u=\beta_v\leq\frac{1}{2}$ by $U_2=U_1$. Since $uv\in \mathcal{B}$, we have $x_u+x_v\geq1$, then $x_v\geq\frac{1}{2}$. Thus $x_v=\frac{1}{2}$ for any $v\in U_1$. For \eqref{s2} and $\rho\geq \sqrt{m}$, we obtain the following  properties:  $E(G[U])=\mathcal{B}$, $x_u=x_v=\frac{1}{2}$ for any $uv\in E(G[U])$, and  $\sum\limits_{w\in W}f(w)=0$. 

Let $U'=\{v\in U\mid d_U(v)\geq1\}$. Then $|U'|\geq\
2$ since $E(G[U])\neq\emptyset$.

    If $W=\emptyset$, then $V(G)=\{u^*\}\cup U$. For any $v\in  U'$, from the characteristic equation at $v$,   it follows that $\frac{1}{2}\rho =1+\frac{d_U(v)}{2}$. Then $G[U']$ is a $(\rho-2)$-regular graph, and thus $\tau(G)\geq e_G[\{u^*\}, U']+\frac{(\rho-2)|U'|}{2}=\frac{\rho|U'|}{2}\geq \rho$.

    If $W\neq\emptyset$, then $f(w)=0$ for any $w\in W$, and thus $d_W(w)=0$ and $d_U(w)(1-x_w)=0$. Since $d_W(w)=0$ for any $w\in W$ and $G$ is connected, we have $d_U(w)>0$, and thus $x_w=1$ for any $w\in W$.  From the characteristic equations at $u^*$ and $w$, it follows that $\rho=\sum\limits_{z\in U}x_z=\sum\limits_{z\in N_U(w)}x_z$. Then $N_G(w)= U$ for any $w\in W$  since $x_z>0$ for any $z\in U$, and thus $G\cong (1+|W|)K_1\vee G[U]$. For any $v\in U'$, we have $\frac{\rho}{2}=1+|W|+\frac{d_{U'}(v)}{2}$. Then $G[U']$ is a $(\rho-2-2|W|)$-regular graph, and thus $\tau(G)=e_G[\{u^*\}\cup W, U']+e(G[U'])=(1+|W|)|U'|+\frac{(\rho-2-2|W|)|U'|}{2}=\frac{\rho |U'|}{2}\geq\rho$.

Combining the above arguments, we have $\tau(G)\geq\rho$ in this case.

\textbf{Case 2.} $U_2 \subsetneq U_1$.

We now complete the proof by contradiction. Suppose that $\tau(G) < \rho$. Since $U_2 \subsetneq U_1$, we have $U_1\setminus U_2\neq \emptyset$.  For any $v\in U_1\setminus U_2$, there exists $v'\in U_1$   such that $vv'\in \mathcal{B}$, $\frac{1}{2}\leq x_{v'}=\beta_v\leq1$       and $\gamma_v>0$. By Lemma \ref{l6}, we have $\tau(G)\geq R(d_U(v'),d_U(v),c_{vv'})$.

\textbf{Subcase 2.1.} $\tau(G)\geq R(d_U(v'),d_U(v),c_{vv'})+1$  for every $v\in U_1\setminus U_2$.

First, we show $\sum\limits_{w\in W\setminus N_G(v)}x_w> 2d_{\mathcal{B}}(v)\gamma_v$ for every $v\in U_1\setminus U_2$. Since $v\in U_1\setminus U_2$, we have $\gamma_v=2\beta_v-1>0$, and thus $\frac{1}{2}<\beta_v\leq1$. By Lemma \ref{l7}, we have 
   $$\rho \beta_v>\beta_v\tau(G)\geq \beta_vR(d_U(v'),d_U(v),c_{vv'})+\beta_v
    \geq d_U(v')+c_{vv'}+2\gamma_vd_U(v)+1,$$
and thus 
\begin{align}
     \sum\limits_{w\in W\setminus N_G(v)}x_w&\geq  \sum\limits_{w\in N_W(v')\setminus N_G(v)}x_w\nonumber\\
    &=\sum\limits_{w\in N_W(v')}x_w-\sum\limits_{w\in N_W(v')\cap N_W(v) }x_w\nonumber\\
    &=\rho(G)\beta_v-1-\sum\limits_{z\in N_U(v') }x_z-\sum\limits_{w\in N_W(v')\cap N_W(v) }x_w\nonumber\\
    &>2\gamma_vd_U(v)\nonumber\\
    &\geq 2d_{\mathcal{B}}(v)\gamma_v\label{r11}
\end{align} for every $v\in U_1\setminus U_2$.
 
% $\sum\limits_{w\in W\setminus N(v)}x_w\geq \rho\beta_v-1-d_U(v')- c_{vv'}$. Combining \eqref{s8}, we obtain 
 %  $ \sum\limits_{w\in W\setminus N(v)}x_w> 2\gamma_vd_U(v)\geq 2d_{\mathcal{B}}(v)\gamma_v$ for every $v\in U_1\setminus U_2$.

  By Lemma \ref{l2}, \eqref{r11} and  Lemma \ref{l5}, we have \begin{align}
      &\sum\limits_{w\in W }f(w)\geq \frac{1}{2}\sum\limits_{v\in U_1\setminus U_2 }(x_v\sum\limits_{w\in W\setminus N_G(v) }x_w)>\sum\limits_{v\in  U_1\setminus U_2 }d_{\mathcal{B}}(v)x_v\gamma_v\nonumber\\&=\sum\limits_{v\in  U }d_{\mathcal{B}}(v)x_v\gamma_v\geq\sum\limits_{uv\in \mathcal{B}}(x_u+x_v-1).\label{e12}
  \end{align}  Combining \eqref{e12} and \eqref{s2}, we have  $\rho<\sqrt{m}$, contradicting $\rho\geq \sqrt{m}$.

  \textbf{Subcase 2.2.} There exists a vertex $y\in U_1\setminus U_2$ such that  $\tau(G)= R(d_U(y'),d_U(y),c_{yy'})$, where $y'\in U_1$,   $yy'\in \mathcal{B}$, $x_{y'}=\beta_y$      and $\gamma_y>0$.

  Let $V^*$ be the set of vertices incident to at least one triangular edge in $G$, and $H=(V^*,T(G))$. Then we have the following claim.

  \noindent\textbf{Claim 1.}  %$E(H)=\{u^*y,u^*y',yy'\}\cup \{u^*u,yu,y'u\mid u\in N_U(y)\cap N_U(y')\}\cup\{u^*z,y'z\mid z\in N_U(y')\setminus ((N_U(y)\cap N_U(y'))\cup \{y\})\}\cup \{yw,y'w\mid w\in N_W(y')\cap N_W(y)\}$, 
  $G[V^*]=H$ and $N_G(v)\cap V^*$ is an independent set of $H$ for any $v\in V(G)\setminus V^*$.

  \begin{proof}
   Without loss of generality, let $d_U(y')\geq d_U(y)$. Then $\tau(G)= R(d_U(y'),d_U(y),c_{yy'})=2d_U(y')+d_U(y)+2c_{yy'}$.   Since \eqref{b6}, we have \begin{align*}
       \tau(G)&\geq | N_U(y)\cup N_U(y')|+(|N_U(y)|+|N_U(y')|-1)+2c_{yy'}\\&\geq d_U(y')+1+(d_U(y)+d_U(y')-1)+2c_{yy'}\\&=R(d_U(y'),d_U(y),c_{yy'})=\tau(G). 
   \end{align*} Thus, $| N_U(y)\cup N_U(y')|=d_U(y')+1$, which implies $N_U(y)\setminus \{y'\}\subseteq N_U(y')$. Clearly, we have \begin{align*}
       E(H)=&\{u^*y,u^*y',yy'\}\cup\{u^*z,y'z\mid z\in N_U(y')\setminus ((N_U(y)\cap N_U(y'))\cup \{y\})\}\\&\cup \{u^*u,yu,y'u\mid u\in N_U(y)\cap N_U(y')\}\cup \{yw,y'w\mid w\in N_W(y')\cap N_W(y)\}
   \end{align*} and $d_{H}(y')=|V^*|-1$.

      Now we show $G[V^*]=H$. Otherwise, there exists $y_1y_2\in E(G[V^*])\setminus E(H)$. If $y_1=y'$ or $y_2=y'$, then $y_1y_2$ is a triangular edge, a contradiction. If $y_1\neq y'$ and $y_2\neq y'$, then $G[\{y',y_1,y_2\}]$ is a triangle, and thus $y_1y_2$ is a triangular edge, a contradiction.

      Finally, we show that $N_G(v)\cap V^*$ is an independent set of $H$ for any $v\in V(G)\setminus V^*$. Otherwise, there exists $v_1\in V(G)\setminus V^*$ such that $G[N_G(v_1)\cap V^*]$ contains an edge $y_3y_4$. Clearly,  $G[\{v_1,y_3,y_4\}]$ is a triangle, and thus $v_1y_3$ is a triangular edge, a contradiction.

     Combining the above arguments completes the proof of Claim 1.
  \end{proof}

 	In what follows, we denote by $I_n$ the $n \times n$ identity matrix, and by $\textbf{j}_n$ the all-one vector of order $n$. In particular, for a vertex set $V$ of order $n$ and a subset $V' \subseteq V$, let $\textbf{j}_{n,V'}$ be the $n$-dimensional vector whose entries corresponding to the vertices in $V'$ are equal to $1$, and all other entries are $0$. When there is no risk of confusion, we write $\textbf{j}$ for $\textbf{j}_{n}$, $I$  for $I_n$  and $\textbf{j}_{V'}$ for $\textbf{j}_{n,V'}$.
	
	Let  $V_z=N_G(z)\cap V^*$ for any $z\in V(G)\setminus V^*$. Recall that $H=(V^*,T(G))$.   By Claim 1, we have $G[V^*]=H$, and $V_z$ is an independent set of $H$ for every $z\in V(G)\setminus V^*$. Therefore,
	$e(H)=\tau(H)=\tau(G)$. 
	Since $yy'\in \mathcal{B}$, the graph $G$ contains the triangle $G[\{u^{*},y,y'\}]$, and hence $\tau(G)\ge 3$. 	Let   $\overline{V_z}=V(H)\setminus V_z$, and $\textbf{p}_z=(p_{v,z})_{v\in V(H)}$ with 
	
	$$p_{v,z}=
	\begin{cases}
		\rho, & \text{if } v\in V_z,\\[1mm]
		d_{V_z}(v)+\dfrac{1}{2}d_{\overline{V_z}}(v), & \text{if } v\in \overline{V_z}.
	\end{cases}$$
	Furthermore, for convenience, set $M=\rho I-A(H)$.

\medskip
\noindent
\textbf{Claim 2.} 
$\textbf{j}^{T}\textbf{p}_z=\rho|V_z|+e(H)$ and $M\textbf{p}_z\ge (\rho^2-e(H))\textbf{j}_{V_z}$. 
Furthermore, there exists $v^*\in V(H)$ such that $(M\textbf{p}_z)_{v^*}> ((\rho^2-e(H))\textbf{j}_{V_z})_{v^*}$.

\begin{proof}

	Since $V_z$ is an independent set, we have
	$$
	\sum_{v\in \overline{V_z}}p_v
	=\sum_{v\in \overline{V_z}}d_{V_z}(v)
	+\frac{1}{2}\sum_{v\in \overline{V_z}}d_{\overline{V_z}}(v)
	=e_H(V_z,\overline{V_z})+e(H[\overline{V_z}])
	=e(H).$$
	Consequently, $\textbf{j}^{T}\textbf{p}_z=\rho|V_z|+e(H)$.
	
	%We next show that
	%$p_u\le \frac{e(F)}{2}$
	%for every $u\in \overline{V_z}$. For each $i\in N_B(u)$, the edge $ui$ is contained in a triangle $F[\{u,i,w_i\}]$. Since $V_z$ is an independent set, we have $w_i\in \overline{V_z}$. Moreover, the $d_B(u)$ edges $iw_i$, where $i\in N_B(u)$, are pairwise distinct and none of them is incident to $u$. Therefore,
	%$e(F)\ge d_F(u)+d_B(u)=2d_B(u)+d_{\overline{V_z}}(u),$
	%which implies that
	%$p_u=d_B(u)+\frac{1}{2}d_{\overline{V_z}}(u)\le \frac{e(F)}{2}.$

	Let $v\in V(H)$. Now we show $(M\textbf{p}_z)_v\geq ((\rho^2-e(H))\textbf{j}_{V_z})_v$ for any $v\in V(H)$.
	
	 If $v\in V_z$, then $$(A(H)\textbf{p}_z)_v=\sum\limits_{u\in N_H(v)\cap \overline{V_z}}(d_{V_z}(u)+\frac{1}{2}d_{\overline{V_z}}(u))\leq e_H(V_z,\overline{V_z})+e(H[\overline{V_z}])\leq e(H),$$ and thus $(M\textbf{p}_z)_v\geq \rho^2-e(H)=((\rho^2-e(H))\textbf{j}_{V_z})_v$.
	
	If $v\in \overline{V_z}$, then \begin{align}
		(M\textbf{p}_z)_v&=\rho d_{V_z}(v)+\frac{\rho}{2}d_{\overline{V_z}}(v)-(d_{V_z}(v)\rho+\sum\limits_{u\in N_{\overline{V_z}}(v)}(d_{V_z}(u)+\frac{1}{2}d_{\overline{V_z}}(u)))\nonumber\\&=\frac{\rho}{2}d_{\overline{V_z}}(v)-\sum\limits_{u\in N_{\overline{V_z}}(v)}(d_{V_z}(u)+\frac{1}{2}d_{\overline{V_z}}(u)).\label{e21}
	\end{align}  
	In fact, for any $u\in \overline{V_z}$ and $i\in N_{V_z}(u)$, the edge $ui$ is contained in a triangle $H[\{u,i,w_i\}]$, where $w_i\in \overline{V_z}$ since $V_z$ is an independent set. Then $\frac{e(H)}{2}\geq \frac{2d_{V_z}(u)+d_{\overline{V_z}}(u)}{2}=d_{V_z}(u)+\frac{1}{2}d_{\overline{V_z}}(u) $, and thus, by \eqref{e21} and $v\in \overline{V_z}$, we have  $$(M\textbf{p}_z)_v\geq (\frac{\rho}{2}-\frac{e(H)}{2})d_{\overline{V_z}}(v)\geq 0=((\rho^2-e(H))\textbf{j}_{V_z})_v.$$
	
	Therefore, $(M\textbf{p}_z)_v\geq ((\rho^2-e(H))\textbf{j}_{V_z})_v$ for any $v\in V(H)$, which implies that $M\textbf{p}_z\geq (\rho^2-e(H))\textbf{j}_{V_z}$.

	%Let $v\in V(F)$. If $v\in B$, then $N_F(v)\subseteq \overline{B}$, and hence
	%	$\big((rI_{|V(F)|}-A(F))P\big)_v
	%	=r^2-\sum_{u\in N_F(v)}p_u
	%	\ge r^2-\sum_{u\in \overline{B}}p_u
	%	=r^2-e(F)
	%	=\big((r^2-e(F))\mathbf{1}_B\big)_v.$
	%If $v\in \overline{B}$, then
	%	$\big((rI_{|V(F)|}-A(F))P\big)_v
	%	=r\left(d_B(v)+\frac{1}{2}d_{\overline{B}}(v)\right)
	%	-rd_B(v)-\sum_{u\in N_{\overline{B}}(v)}p_u
	%	=\frac{r}{2}d_{\overline{B}}(v)-\sum_{u\in N_{\overline{B}}(v)}p_u
	%	\ge \frac{r-e(F)}{2}d_{\overline{B}}(v)
	%	\ge 0
	%	=\big((r^2-e(F))\mathbf{1}_B\big)_v.$
	%Thus,
	%$(rI_{|V(F)|}-A(F))P\ge (r^2-e(F))\mathbf{1}_B.$
	
	Finally,  since $e(H)=\tau(H)\ge 3$, the graph $H$ contains a triangle. As $V_z$ is an independent set, this triangle contains an edge of $H[\overline{V_z}]$. Hence there exists a vertex $v^{*}\in \overline{V_z}$ such that $d_{\overline{V_z}}(v^{*})>0$. Recall that we are assuming $e(H)=\tau(G)<\rho$. It follows that
	$(M\textbf{p}_z)_{v^{*}}
	\ge \frac{\rho-e(H)}{2}d_{\overline{V_z}}(v^{*})
	>0
	=((\rho^2-e(H))\textbf{j}_{V_z})_{v^{*}}$.
	
	The proof of Claim 2 is completed.
\end{proof}

 %Moreover, every edge of $H$ is contained in a triangle of $H$. Indeed, if $uv\in E(H)=T(G)$, then $uv$ is contained in a triangle of $G$, all three vertices of this triangle lie in $K$, and all three edges belong to $T(G)$. 

For convenience, we  normalize the Perron vector $\textbf{x}=(x_v)_{v\in V(G)}$ so that
$\sum\limits_{v\in V(G)}x_v=1$. 
Recall that we are assuming $\tau(G)<\rho$. Then
$
\rho^2(H)\le 2e(H)=2\tau(G)<\tau^2(G)<\rho^2,
$
and thus $\rho(H)<\rho$. Therefore, by the Neumann series theorem,
$M^{-1}=\frac{1}{\rho}\sum\limits_{k\ge 0}\left(\frac{A(H)}{\rho}\right)^k$
exists and is a non-negative matrix, where $M=\rho I-A(H)$.

Now we show that $V(G)\setminus V^*\ne \emptyset$. Otherwise, by Claim 1, we have $G=G[V^*]=H$, which implies that $\rho=\rho(H)<\rho$, a contradiction.

Let $\textbf{x}^{\mathsf{T}}=(\textbf{x}^{\mathsf{T}}_{V^*},\textbf{x}^{\mathsf{T}}_{V(G)\setminus V^*})$, where $\textbf{x}_{V^*}$ and $\textbf{x}_{V(G)\setminus V^*}$ correspond to $V^*$ and $V(G)\setminus V^*$,  respectively. The characteristic equations at the vertices of $V^*$ give
\begin{equation}
M\textbf{x}_{V^*}=\sum_{z\in V(G)\setminus V^*}x_z\textbf{j}_{V_z}.\label{e11}
\end{equation}

 By Claim 2, we have  $M\textbf{p}_z\ge (\rho^2-\tau(G))\textbf{j}_{V_z}$
and $\textbf{j}^{\mathsf{T}}\textbf{p}_z=\rho|V_z|+\tau(G)$ for every $z\in V(G)\setminus V^*$. 
Furthermore, the first inequality is strict in at least one coordinate. Let 
$\textbf{q}=\sum\limits_{z\in V(G)\setminus V^*}x_z\textbf{p}_z.$
Since $x_z>0$ for every $z\in V(G)\setminus V^*$, it follows from \eqref{e11} that
$$M\textbf{q}\ge (\rho^2-\tau(G))M\textbf{x}_{V^*}
\quad\text{and}\quad
M\textbf{q}\ne (\rho^2-\tau(G))M\textbf{x}_{V^*}.$$
As $M^{-1}$ is a non-negative matrix, we obtain
$$
\textbf{q}-(\rho^2-\tau(G))\textbf{x}_{V^*}
=M^{-1}(M\textbf{q}-(\rho^2-\tau(G))M\textbf{x}_{V^*})\ge \mathbf{0}
\quad\text{and}\quad \textbf{q}-(\rho^2-\tau(G))\textbf{x}_{V^*}\neq \mathbf{0}.$$ Therefore,

\begin{align}
	(\rho^2-\tau(G))\sum\limits_{v\in  V^*}x_v%=(\rho^2-\tau(G))\textbf{j}_{|V^*|}^{\mathsf{T}}\textbf{x}_{V^*}
	<\textbf{j}_{|V^*|}^{\mathsf{T}}\textbf{q}
	=\sum_{z\in V(G)\setminus V^*}x_z\textbf{j}_{|V^*|}^{\mathsf{T}}\textbf{p}_z
	=\sum_{z\in V(G)\setminus V^*}x_z(\rho|V_z|+\tau(G)).\label{e32}
\end{align}

On the other hand, summing the characteristic equations over all vertices of $V^*$, we have
\begin{align}
	\rho\sum_{v\in V^*}x_v
	=\sum_{uu'\in T(G)}(x_u+x_{u'})
	+\sum_{z\in V(G)\setminus V^*}|V_z|x_z.\label{e13}
\end{align}

Combining \eqref{e32}, \eqref{e13} and $\sum\limits_{v\in V(G)}x_v=1$, we obtain

\begin{align}
	\rho\sum_{uu'\in T(G)}(x_u+x_{u'})
	&=\rho^2\sum_{v\in V^*}x_v
	-\rho\sum_{z\in V(G)\setminus V^*}|V_z|x_z\nonumber\\
	&<\tau(G)\sum_{v\in V^*}x_v
	+\tau(G)\sum_{z\in V(G)\setminus V^*}x_z\nonumber\\
	&=\tau(G).\label{e14}
\end{align}

For any $v_1v_2\in E(G)\setminus T(G)$, we have $N_G(v_1)\cap N_G(v_2)=\emptyset$. Therefore,
$\rho(x_{v_1}+x_{v_2})
=\sum\limits_{w\in N_G(v_1)}x_w+\sum\limits_{w\in N_G(v_2)}x_w
\le 1.$
Since $|E(G)\setminus T(G)|=m-\tau(G)$, it follows that
\begin{align}
	\rho\sum_{v_1v_2\in E(G)\setminus T(G)}(x_{v_1}+x_{v_2})
	\le m-\tau(G).\label{e15}
\end{align}

Combining \eqref{e14} and \eqref{e15}, we have
	$\rho^2
	=\rho\sum\limits_{uv\in T(G)}(x_u+x_v)
	+\rho\sum\limits_{v_1v_2\in E(G)\setminus T(G)}(x_{v_1}+x_{v_2})\\
	<m$,
which contradicts the fact that $\rho\ge \sqrt{m}$. 

Combining the above arguments, we have $\tau(G)\ge\rho$ in Case 2.

The proof of Theorem \ref{t2} is completed by Case 1 and Case 2. $\hfill\square$

\vspace{0.5cm}

By Theorem \ref{t2}, we immediately obtain the following corollary.

\begin{corollary}
     Every  Nosal graph $G$ satisfies $\tau(G)\geq\rho(G)$.
\end{corollary}

\begin{remark}
The proof of Theorem \ref{t2} is inspired by the ideas of Zhai, Li and Lou \cite{zm}.  To prove Theorem \ref{t2}, we localize the graph $G$ around a vertex $u^*$ at which the Perron vector attains its maximum component, and express $\rho^2(G)-e(G)$ as the difference between the excess contributed by internal edges in the neighborhood of $u^*$ and the deficit contributed by the remaining vertices. Since $\rho^2(G)\geq e(G)$, the excess must be at least the deficit. We define two significant vertex sets $U_1$ and $U_2$. If $U_2=U_1$,  then a straightforward spectral analysis gives $\tau(G)\geq \rho(G)$. If $U_2\subsetneq U_1$, we argue by contradiction. Moreover, we prove that if  the local triangular edge-counting inequality is not tight, then under the assumption $\tau(G)<\rho(G)$, the deficit strictly exceeds the excess, yielding a contradiction. If the local triangular edge-counting inequality is tight, we analyze the structural properties of $G$. Subsequently, via a vector comparison argument, we control the total Perron weight of all triangular edges. Combining this with inequalities for non-triangular edges, we obtain $\rho^2(G)<e(G)$, again a contradiction.

    In fact, the proofs of  Theorems \ref{t1} and \ref{t2} are both based on the same principle: by assigning weights to all vertices via the Perron vector of the graph $G$, the global requirement $\rho(G)\geq \sqrt{e(G)}$ is transformed into local information about triangles. 
\end{remark}

%We now prove Theorem \ref{t3}.

%\noindent\textbf{\textit{Proof of Theorem \ref{t3}.}}  Let $G_0$ be a connected component with $\rho(G_0)=\rho(G)$. Then $\rho(G_0)=\rho(G)\geq \sqrt{m}$.

%If $G_0$ is a complete bipartite, then $\rho(G_0)=\sqrt{e(G_0)}$. Thus we have $\sqrt{m}\leq \rho(G_0)=\sqrt{e(G_0)}\leq \sqrt{m}$, which implies that $e(G_0)=m$. Since $G$ without isolated vertices, $G=G_0$ is a  complete bipartite.

%If $G_0$ is not a complete bipartite, then $\tau(G)\geq \tau (G_0)\geq \rho (G)\geq \sqrt{m}$ by Theorem \ref{t2}.

%Combining the above arguments, we complete the proof. $\hfill\square$

\section{Further work}\label{sec-con}

In this paper, we show as a direct corollary of our main results that  every $m$-edge Nosal graph $G$  satisfies $bk(G)>\frac{\rho(G)}{3}>\frac{\sqrt{m}}{3}$ and $\tau(G)\geq \rho(G)>\sqrt{m}$.  Example 2.3 of \cite{ly} provides an $m$-edge Nosal graph $G$ with $bk(G)\leq \frac{\sqrt{m}}{3}+1$. Furthermore, let $H$ be a graph obtained from $K_{4t+3,t}$ by adding an edge within the part of size $4t+3$. Then $H$ is a Nosal graph with $e(H)=4t^2+3t+1$  and  $\tau(H)=2t+1=\sqrt{e(H)+t}$ by a direct computation. Therefore, the constants $\frac{1}{3}$ and $1$ are optimal in 
$bk(G)>\frac{\sqrt{m}}{3}$
and 
$\tau(G)>\sqrt{m}$,
respectively.

A natural research direction is to consider the \textit{generalized book} $B_{s,k}=K_s\vee kK_1$ with $s\ge 2$, where $k\geq1$ is called the size of $B_{s,k}$. Li, Liu and Zhang \cite{ly} proved the following result on the generalized books.
\begin{theorem}{\rm(\!\!\cite{ly})}
    Every $m$-edge graph $G$ with $\rho^2(G) > 2m(1-\frac{1}{s})$ contains  a copy of $B_{s,k}$ of size $k=\Omega_s(\sqrt{m})$. Furthermore, there are such graphs with the largest generalized booksize $O_s(\sqrt{m})$.
\end{theorem}
Thus, a natural question is the following.

\begin{problem}
For  fixed  $s\geq2$, what is the optimal constant $c_s$ such that every $m$-edge graph $G$ with $\rho^2(G) > 2m(1-\frac{1}{s})$ contains a copy of $B_{s,k}$ with $k \ge c_s\sqrt{m}$?
\end{problem}

For $r\geq3$, an \textit{r-joint} is a family of $r$-cliques sharing a common edge. Let $js_r(G)$ be the maximum number of $r$-cliques in an $r$-joint of $G$. Li, Liu and Zhang \cite{ly} proved the following result on $js_r(G)$.

\begin{theorem}{\rm(\!\!\cite{ly})}\label{l13}
		Let $r\geq3$ be an integer, and $G$ be an $m$-edge graph with $\rho^2(G)>2m(1-\frac{1}{r-1})$.	Then $js_r(G)=\Omega_r(m^{\frac{r-2}{2}})$.
	\end{theorem}
    
Let $t\geq 3$, and let $\tau_t(G)$ denote the number of edges contained in  at least one $K_t$ of $G$. Observe that $\tau_3(G)=\tau(G)$.   We have the following result.

\begin{theorem}
    Let  $t\geq4$ be a fixed positive integer, and $G$ be an $m$-edge graph with $\rho^2(G)>2m(1-\frac{1}{t-1})$. Then $\tau_t(G)=\Omega_t(m)$.
\end{theorem}

\begin{proof}
   By Theorem \ref{l13} and  $\rho^2(G)>2m(1-\frac{1}{t-1})$, we have $js_t(G)=\Omega_t(m^{\frac{t-2}{2}})$. Then there exists an edge $uv\in E(G)$ contained in $\Omega_t(m^{\frac{t-2}{2}})$ copies of $K_t$ of $G$. Let $H$ be a graph with $V(H)=N_G(u)\cap N_G(v)$ and $E(H)=\{e\in E(G[N_G(u)\cap N_G(v)])\mid \text{$e$ is contained in some copy of $K_{t-2}$ of $G[N_G(u)\cap N_G(v)]$}\}$. Then $H$ contains $\Omega_t(m^{\frac{t-2}{2}})$ copies of $K_{t-2}$. On the other hand, by the Kruskal-Katona theorem (see \cite{bb}), $H$ contains at most $O_t(e(H)^{\frac{t-2}{2}})$ copies of $K_{t-2}$. Thus we have $e(H)=\Omega_t(m)$. By the definition of $H$, we get $\tau_t(G)=\Omega_t(m)$.
\end{proof}

Hence, another interesting question is the following.

\begin{problem}
For  fixed  $t\geq4$, what is the optimal constant $c_t$ such that every $m$-edge graph $G$ with $\rho^2(G) > 2m(1-\frac{1}{t-1})$ satisfies $ \tau_t(G)\ge c_tm$?
\end{problem}

%Li, Liu and Feng \cite{LLL2022} posed the following conjecture.

%This improves the results of Wang et al.~\cite{Wang24} and Zhao et al.~\cite{Zhao19}. 

%A direct computation shows that Example 2.3 in \cite{ly} provides a graph $G$ with $bk(G)\leq \frac{1}{3}\rho(G)+1$. Furthermore, let $H$ be a graph obtained from $K_{t,t,t}$ by joining all vertices of one part to all vertices of an independent set of size $6t-1$. Then $\rho(H)>\sqrt{9t^2-t}=\sqrt{e(H)}$.  On the other hand, we have $bk(H)=t<\frac{\rho(H)+1}{3}$. Therefore, in view of the above two examples,  the constant $\frac{1}{3}$ is sharp.

%In the proof of Theorem \ref{t1}, we  use Proposition \ref{p1}. This is a local structural information. We believe that this spectral approach could further improve the constant, provided that a global analysis of $U$ in the proof of Theorem \ref{t1} is carried out. However, it remains very difficult to completely solve Problem \ref{P1}.

	\section*{\bf Funding}
	
 This work is  supported by the National Natural Science Foundation of China (Grant Nos. 12371347, 12271337).

\section*{Acknowledgments}

The authors used an AI tool for exploratory discussions during the preparation of this work. All mathematical claims, proofs, and citations were independently reviewed, corrected, and approved by the authors, who take full responsibility for the final version of the paper.
	
	%\section*{Declarations}
	
%	\noindent\textbf{Conflict of interest}\  The authors declare that they have no conflict of interest.
	
	%\vskip 0.5em
	
	%\noindent\textbf{Data availability} \  No data was used for the research described in the article.

	\end{spacing}
\end{document}